\documentclass[12pt]{amsart}
\usepackage{amssymb,amsmath,amsthm,mathrsfs,multirow,xcolor,framed,url}
\usepackage[pdftex,
         pdfauthor={Dandan Chen and Zhiguo Liu},
         pdftitle={New proof of some $q$-orthogonal functions},
         pdfsubject={MATHEMATICS},
         pdfkeywords={$q$-beta integral, $q$-orthogonal functions, $q$-ultraspherical polynomials, Askey-Wilson polynomials},
         pdfproducer={Latex with hyperref},
         pdfcreator={pdflatex}]{hyperref}
\newtheorem{theorem}{Theorem}[section]

\newtheorem{prop}[theorem]{Proposition}

\theoremstyle{definition}

\theoremstyle{remark}

\numberwithin{equation}{section}

\newcommand\nutwid{\overset {\text{\lower 3pt\hbox{$\sim$}}}\nu}

\allowdisplaybreaks

\newcommand\omycite[1]{}

\newcommand{\beqs}{\begin{equation*}}
\newcommand{\eeqs}{\end{equation*}}
\newcommand{\beq}{\begin{equation}}
\newcommand{\eeq}{\end{equation}}
\renewcommand{\MR}[1]{\href{http://www.ams.org/mathscinet-getitem?mr={#1}}{MR{#1}}}

\begin{document}
\title[New proofs of theorems]{New proofs of theorems on q-orthogonal functions}

\author{Dandan Chen}
\address{Department of Mathematics, Shanghai University, People's Republic of China}
\address{Newtouch Center for Mathematics of Shanghai University, Shanghai, People's Republic of China}
\email{mathcdd@shu.edu.cn}
\author{Zhiguo Liu}
\address{School of Mathematical Sciences and Shanghai Key Laboratory of PMMP, East China Normal University,500 Dongchuan Road, Shanghai 200241, P. R. China}
\email{zgliu@math.ecnu.edu.cn;liuzg@hotmail.com}



\date{}



\begin{abstract}
In this paper, we  establish a $q$-integral formula by using the orthogonality relation, and   also provide a new proof of the $q$-orthogonality relation for the continuous $q$-ultraspherical polynomials. A new $q$-beta integral with five parameters is evaluated.
\end{abstract}

\maketitle

\section{Introduction}

In his work on the Dirichlet problem for the Heisenberg group, Greiner \cite{Greiner-CMB-80} showed that each $L_\alpha$-spherical harmonic is a unique linear combination of functions of the form
\begin{align*}
e^{in\psi}\sin^{|n|/2}\theta H_k^{(\alpha,n)}(e^{i\theta});
\end{align*}
 where $H_k^{(\alpha,n)}(e^{i\theta})t^k$ is defined by the generating function
\begin{align*}
(1-te^{-i\theta})^{(n-|n|+\alpha-1)/2}(1-te^{i\theta})^{-(n+|n|+\alpha+1)/2}
=\sum_{k=0}^{\infty}H_k^{(\alpha,n)}(e^{i\theta})t^k
\end{align*}
with $k\in \mathbb{N}$ and $n\in \mathbb{Z}$. Greiner \cite{Greiner-CMB-80} raised the question of whether the functions $H_k^{(\alpha,n)}(e^{i\theta})t^k$ are orthogonal or bi-orthogonal with respect to some complex valued weight functions. Gasper \cite{Gasper-CM-81} explained that this was indeed the case and showed that the more general class of functions $C_k^{(\alpha,\beta)}(e^{i\theta})$ satisfied the orthogonality relation. He also considered $q$-(basic) analogs of $C_k^{(\alpha,\beta)}(e^{i\theta})$ which contain the continuous $q$-ultraspherical polynomials as a special case \cite{Gasper-CM-81}, and provided conditions under which they are orthogonal.

The continuous $q$-ultraspherical polynomials, also known as the Rogers-Askey-Ismail polynomials, constitute a family of orthogonal polynomials introduced by L. J. Rogers  in 1895 \cite{Rogers-PLM-95}. He utilized them to establish the Rogers-Ramanujan identities. Rogers described several important properties of these polynomials, but he  did not establish their orthogonality. The orthogonality of these polynomials was first proved by R. Askey and M. E. H. Ismail  in 1980 \cite{Askey-Ismail-80}.

The continuous $q$-ultraspherical polynomials are defined in \cite{Askey-Ismail-80} by
\begin{align}\label{defn-c-beta}
\frac{(\beta re^{i\theta},\beta re^{-i\theta};q)_\infty}{( re^{i\theta}, re^{-i\theta};q)_\infty}
=\sum_{n=0}^{\infty}C_n(\cos\theta;\beta|q)r^n.
\end{align}
Here and later, we use the standard $q$-series notation, where for $|q|<1$ and $n\in\mathbb{N^{+}}$,
\begin{align*}
(a;q)_0:=1, \quad (a;q)_n:=\prod_{k=0}^{n-1}(1-aq^k), \quad (a;q)_\infty:=\prod_{k=0}^\infty(1-aq^k).
\end{align*}
By \eqref{defn-c-beta}, we can deduce that $C_0(\cos\theta;\beta|q)=1$ and $C_1(\cos\theta;\beta|q)=2\cos\theta(1-\beta)/(1-q)$. Additionally, a recurrence relation for these polynomials is given by
\begin{align*}
(1-q^{n+1})C_{n+1}(\cos\theta;\beta|q)=2\cos\theta(1-\beta q^n)C_n(\cos\theta;\beta|q)-(1-\beta^2 q^{n-1})C_{n-1}(\cos\theta;\beta|q).
\end{align*}

From \eqref{defn-c-beta} and the $q$-binomial theorem \cite[p. 92]{Gasper-Rahman-04}
\begin{align}\label{defn-bi-thm}
\frac{(az;q)_\infty}{(z;q)_\infty}=\sum_{n=0}^{\infty}\frac{(a;q)_n}{(q;q)_n}z^n,\quad \text{for}~|z|<1,
\end{align}
it follows that
\begin{align*}
C_n(\cos\theta;\beta|q)=\sum_{k=0}^{n}\frac{(\beta;q)_k(\beta;q)_{n-k}}{(q;q)_k(q;q)_{n-k}}e^{i(n-2k)\theta}
=\sum_{k=0}^{n}\frac{(\beta;q)_k(\beta;q)_{n-k}}{(q;q)_k(q;q)_{n-k}}\cos(n-2k)\theta.
\end{align*}

Let $\omega_\beta(\cos\theta|q)$ be defined by \cite{Askey-Ismail-80} as
\begin{align*}
\omega_\beta(\cos\theta|q)=\frac{(e^{2i\theta},e^{-2i\theta};q)_\infty}{( \beta e^{2i\theta}, \beta e^{-2i\theta};q)_\infty}.
\end{align*}
\begin{theorem}\label{thm-cq-orth-b}
For $|q|<1$ and $|\beta|<1$, the orthogonality relation for the continuous $q$-ultraspherical polynomials states that
\begin{align}\label{eq-cq-orth-b}
\int_{0}^{\pi}C_m(\cos\theta;\beta|q)C_n(\cos\theta;\beta|q)\omega_\beta(\cos\theta|q)d\theta
=\frac{\delta_{m,n}}{h_n(\beta|q)},
\end{align}
where
\begin{align*}
h_n(\beta|q)=\frac{(q,\beta^2;q)_\infty (q;q)_n (1-\beta q^n)}{2\pi(\beta,\beta q;q)_\infty (\beta^2;q)_n (1-\beta)}
\end{align*}
and the Kronecker delta $\delta_{m,n}$ is defined as $1$ for $m=n$ and $0$ otherwise.
\end{theorem}
For more details regarding the theorem discussed above, the reader may be referred to the works of \cite{Askey-Ismail-80,Askey-Ismail-83,Askey-Wilson-85,Gasper-Rahman-04,Ismail-05,KLS-10}.
Let the continuous $q$-Hermite polynomials $H_n(x|q)$ be defined by $H_n(x|q)=(q;q)_nC_n(x;0|q)$. When $\beta=0$, the orthogonality relation in \eqref{eq-cq-orth-b} simplifies to:
\begin{align*}
\int_{0}^{\pi}(e^{2i\theta},e^{-2i\theta};q)_\infty H_n(\cos\theta|q)H_m(\cos\theta|q)d\theta
=\frac{2\pi(q;q)_m}{(q;q)_\infty}\delta_{m,n}.
\end{align*}

We have derived the following $q$-integral formula by using the orthogonality relation of the continuous $q$-ultraspherical polynomials.
\begin{theorem}\label{thm-cq-mn}
For $|q|<1$ and $|\beta|<1$, we have
\begin{align*}
&\int_{0}^{\pi}C_m(\cos\theta;\gamma|q)C_n(\cos\theta;\beta|q)\omega_\beta(\cos\theta|q)d\theta\\
=&\bigg\{\begin{matrix}
\frac{(1-\beta q^n)\beta^{\frac{m-n}{2}}(\gamma/\beta;q)_{\frac{m-n}{2}}(\gamma;q)_{\frac{m+n}{2}}}
{(1-\beta)h_n(\beta|q)(q;q)_{\frac{m-n}{2}}(q\beta;q)_{\frac{m+n}{2}}},&\textrm{if} \qquad m\equiv n\pmod 2\\
0,&\textrm{if}\qquad m\neq n\pmod 2.
\end{matrix}
\end{align*}
\end{theorem}

It is natural to consider, for a $q$-analog of $C_k^{(\alpha,\beta)}(e^{i\theta})$, the functions $C_k^{(\alpha,\beta)}(e^{i\theta};q)$ defined by \cite[Eq. (3.2)]{Gasper-CM-81}
\begin{align}\label{defn-cq-ab}
\frac{(\alpha te^{i\theta},\beta te^{-i\theta};q)_\infty}
{( te^{i\theta}, te^{-i\theta};q)_\infty}
=\sum_{n=0}^{\infty}C_n^{(\alpha,\beta)}(e^{i\theta};q)t^n.
\end{align}

From \eqref{defn-bi-thm} and \eqref{defn-cq-ab}, it follows that
\begin{align*}
C_n^{(\beta,\beta)}(e^{i\theta};q)=\sum_{k=0}^{n}\frac{(\beta;q)_k(\beta;q)_{n-k}}{(q;q)_k(q;q)_{n-k}}e^{i(n-2k)\theta}
=\sum_{k=0}^{n}\frac{(\alpha;q)_k(\beta;q)_{n-k}}{(q;q)_k(q;q)_{n-k}}\cos(n-2k)\theta.
\end{align*}
And we have $C_n^{(\beta,\beta)}(e^{i\theta};q)=C_n(\cos\theta;\beta|q)$. Hence $C_n^{(\alpha,\beta)}(e^{i\theta};q)$ are extensions of the continuous $q$-ultraspherical polynomials.

In this paper, we provide a new proof of the following orthogonality relation for the $q$-functions $C_n^{(\alpha,\beta)}(e^{i\theta};q)$.
\begin{theorem} $\mathrm{(}$\cite{Gasper-CM-81}$\mathrm{)}$\label{thm-cq-orth-ab}
For $|q|<1$, $|\alpha|<1$, and $|\beta|<1$, we have
\begin{align}\label{eq-cq-orth-ab}
&\int_{0}^{2\pi}C_m^{(\alpha,\beta)}(e^{i\theta};q)C_n^{(\alpha,\beta)}(e^{i\theta};q)
\omega^{(\alpha,\beta)}(\cos\theta|q)d\theta\nonumber\\
=&\frac{2\pi(\alpha,\beta;q)_\infty}{(q,\alpha\beta;q)_\infty}
\left(\frac{1}{1-\alpha q^n}+\frac{1}{1-\beta q^n}\right)\frac{(\alpha\beta;q)_n}{(q;q)_n}\delta_{m,n},
\end{align}
where
\begin{align*}
\omega^{(\alpha,\beta)}(\cos\theta|q)=\frac{(e^{2i\theta},e^{-2i\theta};q)_\infty}{( \alpha e^{2i\theta}, \beta e^{-2i\theta};q)_\infty}.
\end{align*}
\end{theorem}

Letting $\alpha=\beta$ in Theorem \ref{thm-cq-orth-ab} and making some simple calculations, we find that the Theorem \ref{thm-cq-orth-ab} reduces to Theorem \ref{thm-cq-orth-b}. Hence Theorem \ref{thm-cq-orth-ab} is an extension of the orthogonality relation for the continuous $q$-ultraspherical polynomials. Using Theorem \ref{thm-cq-orth-ab}, we can also derive the following new $q$-beta integral formula with five parameters $\alpha,\beta,s,t$, and $q$.
\begin{theorem}\label{thm-cq-orth-ab-1}
For $\max\{|q|,|\alpha|,|\beta|,|s|,|t|\}<1$, we have
\begin{align*}
&\int_{0}^{2\pi}\frac{(\alpha t e^{i\theta},\beta t e^{-i\theta},\alpha s e^{i\theta},\beta s e^{-i\theta}, e^{2i\theta}, e^{-2i\theta};q)_\infty}
{(t e^{i\theta}, t e^{-i\theta}, s e^{i\theta}, s e^{-i\theta}, \alpha e^{2i\theta}, \beta e^{-2i\theta};q)_\infty}d\theta\nonumber\\
=&\frac{2\pi(\alpha,\beta;q)_\infty}{(q,\alpha\beta;q)_\infty}
\sum_{n=0}^{\infty}\left(\frac{1}{1-\alpha q^n}+\frac{1}{1-\beta q^n}\right)\frac{(\alpha\beta;q)_n}{(q;q)_n}(st)^n.
\end{align*}
\end{theorem}
The paper is organized as follows. In Section \ref{sec-pre}, we recall some asymptotic properties of the $q$-functions $C_n^{(\alpha,\beta)}(e^{i\theta};q)$. We give the proof of Theorems \ref{thm-cq-orth-ab}--\ref{thm-cq-orth-ab-1} in Section \ref{proof-cq-ab}.  In Section \ref{note}, we  prove Theorem \ref{thm-cq-mn} and establish the relation identity for the continuous $q$-ultraspherical polynomials.

\section{Some asymptotic properties of the $q$-functions $C_n^{(\alpha,\beta)}(e^{i\theta};q)$}\label{sec-pre}
It is easy to know that  $(\alpha;q)_n>0$ when $|\alpha|<1$ for nonnegative integer $n$. With the fact that $|e^{i(n-2k)\theta}|=1$, we obtain that
\begin{align*}
|C_n^{(\alpha,\beta)}(e^{i\theta};q)|\leq C_n^{(\alpha,\beta)}(1;q),\quad \textrm{with}  -1<\alpha,\beta<1\textrm{and} -1<q<1.
\end{align*}
For any complex numbers $\alpha$ and $\beta$, the homogeneous polynomials $\Phi_n^{(\alpha,\beta)}(x,y|q)$ are given by \cite[Definition 1.4]{Liu-17}
\begin{align*}
\Phi_n^{(\alpha,\beta)}(x,y|q)=\sum_{k=0}^{n}\left[\begin{matrix}n\\k
\end{matrix}\right]_q(\alpha;q)_k(\beta;q)_{n-k}x^ky^{n-k}.
\end{align*}
where the $q$-binomial coefficients are given by
\begin{align*}
\left[\begin{matrix}n\\k
\end{matrix}\right]_q
=\frac{(q;q)_n}{(q;q)_k(q;q)_{n-k}}.
\end{align*}
We find that
\begin{align}\label{eq-Cq-Phi}
\Phi_{n}^{(\alpha,\beta)}(e^{i\theta},e^{-i\theta}|q)=(q;q)_nC_n^{(\alpha,\beta)}(e^{i\theta};q).
\end{align}
Recall that \cite[Proposition 3.2]{Liu-17},
\begin{align}\label{eq-Phi-xy}
\sum_{n=0}^{\infty}\Phi_n^{(\alpha,\beta)}(x,y|q)\frac{t^n}{(q;q)_n}
=\frac{(\alpha xt,\beta yt;q)_\infty}{(xt,yt;q)_\infty}.
\end{align}
Setting $x=y=1$ in \eqref{eq-Phi-xy}, we have
\begin{align*}
\frac{(\alpha t,\beta t;q)_\infty}{(t;q)_\infty^2}=\sum_{n=0}^{\infty}C_n^{(\alpha,\beta)}(1;q)t^n.
\end{align*}
The radius of convergence of the above power series centered at the origin is the distance to the nearest singularity in the complex plane, namely at 1. Thus, the radius of convergence is 1. Consequently, we have
\begin{align*}
\limsup_{n\rightarrow\infty}C_n^{(\alpha,\beta)}(1;q)^{1/n}=1.
\end{align*}

\begin{prop}\label{prop-cq-condi}
Let $|q|<1$, $|\alpha|<1$, and $|\beta|<1$. Then for any nonnegative integer $k$, and any $t$ such that $|t|<1$, the series
\begin{align*}
\sum_{n=0}^{\infty}C_{n+k}^{(\alpha,\beta)}(e^{i\theta};q)C_{n}^{(\alpha,\beta)}(e^{i\theta};q)
\frac{(q;q)_{n+k}}{(\alpha\beta;q)_{n+k}}t^n
\end{align*}
converges uniformly and absolutely on $[0,2\pi]$.
\end{prop}

\begin{proof}
For $-1<q<1$ and $-1<\alpha,\beta<1$, using the triangular inequality, we easily find that
\begin{align*}
\left|\frac{(q;q)_{n+k}}{(\alpha\beta;q)_{n+k}}\right|\leq\frac{(-|q|;|q|)_\infty}{(|\alpha\beta|;|q|)_\infty}.
\end{align*}
With the help of the above inequality and using the triangular inequality again, we deduce that
\begin{align*}
\left|\sum_{n=0}^{\infty}C_{n+k}^{(\alpha,\beta)}(e^{i\theta};q)C_{n}^{(\alpha,\beta)}(e^{i\theta};q)
\frac{(q;q)_{n+k}}{(\alpha\beta;q)_{n+k}}t^n\right|
\leq\frac{(-|q|;|q|)_\infty}{(|\alpha\beta|;|q|)_\infty}\sum_{n=0}^{\infty}C_{n+k}^{(\alpha,\beta)}(1;q)C_{n}^{(\alpha,\beta)}(1;q)|t|^n.
\end{align*}
Using the above equation and the root test, we know that the series on the right-hand side  converges  for $|t|<1$.
Therefore, the series in Proposition \ref{prop-cq-condi} converges uniformly and absolutely
on $[0,2\pi]$. We thus complete the proof of the proposition by analytic continuation.
\end{proof}

\section{The proof of Theorems \ref{thm-cq-orth-ab}--\ref{thm-cq-orth-ab-1}}\label{proof-cq-ab}
The Jackson $q$-integral of the function $f(x)$ from $a$ to $b$ is defined by \cite[p. 23]{Gasper-Rahman-04}
\begin{align*}
\int_{a}^{b}f(x)d_qx=(1-q)b\sum_{n=0}^{\infty}q^nf(bq^n)-(1-q)a\sum_{n=0}^{\infty}q^nf(aq^n).
\end{align*}
Using the Jackson $q$-integral, Al-Salam and Verma \cite {AlSalam-Verma-PAMS-82} rewrote Sears' nonterminating
extension of the $q$-Saalsch\"{u}tz summation in the following beautiful form.

\begin{prop}
If there are no zero factors in the denominator of the integral and \\
$\max\{|a|,|b|,|cx|,|cy|,|ax/y|,|by/x|\}<1$, then we have
\begin{align*}
\int_{x}^{y}\frac{(qz/x,qa/y,abcz;q)_\infty}{(az/y,bz/x,cz;q)_\infty}d_qz
=\frac{(1-q)y(q,x/y,qy/x,ab,acx,bcy;q)_\infty}{(ax/y,by/x,a,b,cx,cy;q)_\infty}.
\end{align*}
\end{prop}

Using this proposition, we can find the following $q$-integral representation
for $\Phi_n^{(a,b)}(x,y|q)$ \cite[Proposition 5.1]{Liu-17}.
\begin{prop}
If there are no zero factors in the denominator of the integral, then we have
\begin{align}\label{eq-Phi-int}
\Phi_n^{(a,b)}(x,y|q)
=\frac{(ab;q)_n(a,b,by/x,ax/y;q)_\infty}{(1-q)y(q,ab,x/y,qy/x;q)_\infty}
\int_{x}^{y}\frac{(qz/x,qz/y;q)_\infty z^n}{(bz/x,az/y;q)_\infty}d_qz.
\end{align}
\end{prop}
\subsection{The proof of Theorem \ref{thm-cq-orth-ab}}
\begin{proof}
Let $k$ be a nonnegative integer. Replacing $n$ by $n+k$ \eqref{eq-Phi-int}
and then dividing both sides by $(ab;q)_n$ in the resulting equation, we obtain
\begin{align*}
\frac{\Phi_{n+k}^{(a,b)}(x,y|q)}{(ab;q)_{n+k}}
=\frac{(a,b,by/x,ax/y;q)_\infty}{(1-q)y(q,ab,x/y,qy/x;q)_\infty}
\int_{x}^{y}\frac{(qz/x,qz/y;q)_\infty z^{n+k}}{(bz/x,az/y;q)_\infty}d_qz.
\end{align*}
Multiplying both sides by $\Phi_n^{(c,d)}(u,v|q)t^n/(q;q)_n$
and then summing the resulting equation over n from $0$ to $\infty$, we arrive at
\begin{align*}
&\sum_{n=0}^{\infty}\Phi_{n+k}^{(a,b)}(x,y|q)\Phi_{n}^{(c,d)}(u,v|q)
\frac{t^n}{(q;q)_n(ab;q)_{n+k}}\\
=&\frac{(a,b,by/x,ax/y;q)_\infty}{(1-q)y(q,ab,x/y,qy/x;q)_\infty}
\int_{x}^{y}\frac{(qz/x,qz/y,cutz,dvtz;q)_\infty z^k}{(bz/x,az/y,utz,vtz;q)_\infty}d_qz.
\end{align*}
Setting $x=u=e^{i\theta}$,$y=v=e^{-i\theta}$, $a=c=\alpha$, and $b=d=\beta$ in the above equation, we deduce that for $|q|<1, |\alpha|<1, |\beta|<1$, and $|t|<1$,
\begin{align*}
&\frac{(q,\alpha\beta,e^{2i\theta},e^{-2i\theta};q)_\infty}
{(\alpha,\beta,\alpha e^{2i\theta},\beta e^{-2i\theta};q)_\infty}
\sum_{n=0}^{\infty}C_{n+k}^{(\alpha,\beta)}(e^{i\theta};q)C_{n}^{(\alpha,\beta)}(e^{i\theta};q)
\frac{(q;q)_{n+k}}{(\alpha\beta;q)_{n+k}}t^n\\
=&\frac{e^{i\theta}(1-e^{-2i\theta})}{(1-q)}
\int_{e^{i\theta}}^{e^{-i\theta}}\frac{(qze^{i\theta},qze^{-i\theta},\alpha tze^{i\theta},\beta tze^{-i\theta};q)_\infty z^k}{(\alpha ze^{i\theta},\beta ze^{-i\theta}, tze^{i\theta}, tze^{-i\theta};q)_\infty}d_qz\\
=&(1-e^{-2i\theta})e^{-ik\theta}\sum_{n=0}^{\infty}\frac{(q^{n+1},\alpha tq^n,q^{n+1}e^{-2i\theta},\beta tq^ne^{-2i\theta};q)_\infty}{(\alpha q^{n}, tq^n,\beta q^{n}e^{-2i\theta}, tq^ne^{-2i\theta};q)_\infty}q^{n(k+1)}\\
&\quad +(1-e^{2i\theta})e^{ik\theta}\sum_{n=0}^{\infty}\frac{(q^{n+1},\beta tq^n,q^{n+1}e^{2i\theta},\alpha tq^ne^{2i\theta};q)_\infty}{(\beta q^{n}, tq^n,\alpha q^{n}e^{2i\theta}, tq^ne^{2i\theta};q)_\infty}q^{n(k+1)}.
\end{align*}
Inspecting the first series in the above equation, we see that this series can be expanded
in terms of $\{e^{-ni\theta}\}_{n=k}^\infty$. Hence there exists a sequence $\{\lambda_n\}_{n=k}^\infty$
independent of $\theta$ such that
\begin{align}\label{eq-sum-lambda}
(1-e^{-2i\theta})e^{-ik\theta}\sum_{n=0}^{\infty}\frac{(q^{n+1},\alpha tq^n,q^{n+1}e^{-2i\theta},\beta tq^ne^{-2i\theta};q)_\infty}{(\alpha q^{n}, tq^n,\beta q^{n}e^{-2i\theta}, tq^ne^{-2i\theta};q)_\infty}q^{n(k+1)}
=\sum_{n=k}^{\infty}\lambda_ne^{-in\theta}.
\end{align}
Similarly, there exists a sequence $\{\delta_n\}_{n=k}^\infty$ independent of $\theta$ such that
\begin{align*}
(1-e^{2i\theta})e^{ik\theta}\sum_{n=0}^{\infty}\frac{(q^{n+1},\beta tq^n,q^{n+1}e^{2i\theta},\alpha tq^ne^{2i\theta};q)_\infty}{(\beta q^{n}, tq^n,\alpha q^{n}e^{2i\theta}, tq^ne^{2i\theta};q)_\infty}q^{n(k+1)}
=\sum_{n=k}^{\infty}\delta_ne^{in\theta}.
\end{align*}
Adding the above two equations together yields
\begin{align*}
&\frac{(q,\alpha\beta,e^{2i\theta},e^{-2i\theta};q)_\infty}{(\alpha,\beta,\alpha e^{2i\theta},\beta e^{-2i\theta};q)_\infty}
\sum_{n=0}^{\infty}C_{n+k}^{(\alpha,\beta)}(e^{i\theta};q)C_{n}^{(\alpha,\beta)}(e^{i\theta};q)
\frac{(q;q)_{n+k}}{(\alpha\beta;q)_{n+k}}t^n\\
=&\sum_{n=k}^{\infty}\lambda_ne^{-in\theta}+\sum_{n=k}^{\infty}\delta_ne^{in\theta}.
\end{align*}
By proposition \ref{prop-cq-condi}, we know that the series on the left-hand side converges uniformly and absolutely on $[0,2\pi]$. If we integrate the above equation with respect to $\theta$ over $[0,2\pi]$ term by term, and using the known fact
\begin{align*}
\int_{0}^{2\pi}e^{in\theta}d\theta=2\pi\delta_{n,0},
\end{align*}
we deduce that for $k>0$,
\begin{align*}
\sum_{n=0}^{\infty}\left\{\int_{0}^{2\pi}C_{n+k}^{\alpha,\beta}(e^{i\theta};q)C_{n}^{\alpha,\beta}(e^{i\theta};q)
\frac{(e^{2i\theta},e^{-2i\theta};q)_\infty}{(\alpha e^{2i\theta},\beta e^{-2i\theta};q)_\infty}d\theta\right\}
\frac{(q;q)_{n+k}}{(\alpha\beta;q)_{n+k}}t^n=0.
\end{align*}
By equating the coefficients of powers of $t$ on  both sides of the above equation, we find that for $k>0$,
\begin{align}\label{eq-Cq-nk}
\int_{0}^{2\pi}C_{n+k}^{(\alpha,\beta)}(e^{i\theta};q)C_{n}^{(\alpha,\beta)}(e^{i\theta};q)
\frac{(e^{2i\theta},e^{-2i\theta};q)_\infty}{(\alpha e^{2i\theta},\beta e^{-2i\theta};q)_\infty}d\theta=0.
\end{align}
For $k=0$, we have
\begin{align*}
\sum_{n=0}^{\infty}\left\{\int_{0}^{2\pi}C_{n}^{(\alpha,\beta)}(e^{i\theta};q)^2
\frac{(e^{2i\theta},e^{-2i\theta};q)_\infty}{(\alpha e^{2i\theta},\beta e^{-2i\theta};q)_\infty}d\theta\right\}
\frac{(q;q)_{n}}{(\alpha\beta;q)_{n}}t^n
=2\pi\frac{(\alpha,\beta;q)_\infty}{(q,\alpha\beta;q)_\infty}(\lambda_0+\delta_0).
\end{align*}
Now we begin to compute $\lambda_0$ and $\delta_0$. Inspecting the series in \eqref{eq-sum-lambda}, we find that the constant term of the fourier expansion of this series is
\begin{align*}
\lambda_0=\sum_{n=0}^{\infty}\frac{(q^{n+1},\alpha tq^n;q)_\infty}{(\alpha q^n, tq^n;q)_\infty}q^n.
\end{align*}
Applying the $q$-binomial theorem \eqref{defn-bi-thm} to the right-hand side of the above equation, we find that
\begin{align*}
\lambda_0=\sum_{n=0}^{\infty}\frac{t^n}{1-\alpha q^n}.
\end{align*}
In the same way, we have
\begin{align*}
\delta_0=\sum_{n=0}^{\infty}\frac{t^n}{1-\beta q^n}.
\end{align*}
It follows that
\begin{align*}
&\sum_{n=0}^{\infty}\left\{\int_{0}^{2\pi}C_{n}^{(\alpha,\beta)}(e^{i\theta};q)^2
\frac{(e^{2i\theta},e^{-2i\theta};q)_\infty}{(\alpha e^{2i\theta},\beta e^{-2i\theta};q)_\infty}d\theta\right\}
\frac{(q;q)_{n}}{(\alpha\beta;q)_{n}}t^n\\
=&2\pi\frac{(\alpha,\beta;q)_\infty}{(q,\alpha\beta;q)_\infty}
\sum_{n=0}^{\infty}\left(\frac{1}{1-\alpha q^n}+\frac{1}{1-\beta q^n}\right)t^n.
\end{align*}
Equating the coefficients of $t^n$ on both sides of the above equation, we arrive at
\begin{align*}
\int_{0}^{2\pi}C_{n}^{(\alpha,\beta)}(e^{i\theta};q)^2
\frac{(e^{2i\theta},e^{-2i\theta};q)_\infty}{(\alpha e^{2i\theta},\beta e^{-2i\theta};q)_\infty}d\theta
=2\pi\frac{(\alpha,\beta;q)_\infty(\alpha\beta;q)_n}{(q,\alpha\beta;q)_\infty(q;q)_n}
\left(\frac{1}{1-\alpha q^n}+\frac{1}{1-\beta q^n}\right).
\end{align*}
Combining the above equation with \eqref{eq-Cq-nk}, we complete the proof of Theorem \ref{thm-cq-orth-ab}.
\end{proof}
\subsection{The proof of Theorem \ref{thm-cq-orth-ab-1}}
\begin{proof}
Multiplying  on both sides of \eqref{eq-cq-orth-ab} by $s^mt^n$, we have
\begin{align*}
&\int_{0}^{2\pi}C_{m}^{(\alpha,\beta)}(e^{i\theta};q)C_{n}^{(\alpha,\beta)}(e^{-i\theta};q)
\frac{(e^{2i\theta},e^{-2i\theta};q)_\infty}{(\alpha e^{2i\theta},\beta e^{-2i\theta};q)_\infty}s^mt^nd\theta\\
=&2\pi\frac{(\alpha,\beta;q)_\infty}{(q,\alpha\beta;q)_\infty}
\left(\frac{1}{1-\alpha q^n}+\frac{1}{1-\beta q^n}\right)\frac{(\alpha\beta;q)_n}{(q;q)_n}s^mt^n\delta_{m,n}.
\end{align*}
Summing the above equation over $m$ and $n$ from $m=0$ to $m=\infty$ and $n=0$ to $n=\infty$, we conclude that
\begin{align*}
&\int_{0}^{2\pi}\left(\sum_{m=0}^{\infty}C_{m}^{(\alpha,\beta)}(e^{i\theta};q)s^m\right)
\left(\sum_{n=0}^{\infty}C_{n}^{(\alpha,\beta)}(e^{-i\theta};q)t^n\right)
\frac{(e^{2i\theta},e^{-2i\theta};q)_\infty}{(\alpha e^{2i\theta},\beta e^{-2i\theta};q)_\infty}d\theta\\
=&2\pi\frac{(\alpha,\beta;q)_\infty}{(q,\alpha\beta;q)_\infty}
\sum_{n=0}^{\infty}\left(\frac{1}{1-\alpha q^n}+\frac{1}{1-\beta q^n}\right)\frac{(\alpha\beta;q)_n}{(q;q)_n}(st)^n.
\end{align*}
Applying \eqref{defn-cq-ab} to the left-hand side of the above equation, we complete the proof of Theorem \ref{thm-cq-orth-ab-1}.
\end{proof}

\section{A note on the continuous $q$-ultraspherical polynomials}\label{note}
The $q$-hypergeometric series, denoted as ${}_{r+1}\phi_{r}(\cdot)$, 
is defined by the basic hypergeometric series \cite{Gasper-Rahman-04}
\begin{align*}
{}_{r+1}\phi_{r}\bigg(\genfrac{}{}{0pt}{}{a_1,a_2,\cdots,a_{r+1}}{b_1,b_2,\cdots,b_{r}};q,z\bigg)
=\sum_{n=0}^{\infty}\frac{(a_1,a_2,\cdots,a_{r+1};q)_n}{(q,b_1,b_2,\cdots,b_{r};q)_n}z^n,
\end{align*}
and we will also use the notation ${}_{r+1}W_{r}(\cdot)$ to denote the very-well poised series \cite{Gasper-Rahman-04}
\begin{align*}
{}_{r+1}W_{r}(a_1;a_2,\cdots,a_{r-1};q,z)
={}_{r+1}\phi_{r}\bigg(\genfrac{}{}{0pt}{}{a_1,q\sqrt{a_1},-q\sqrt{a_1},a_2,\cdots,a_{r-1}}{\sqrt{a_1},-\sqrt{a_1},qa_1/a_2,\cdots,qa_1/a_{r-1}};q,z\bigg).
\end{align*}
The Rogers ${}_{6}\phi_{5}(\cdot)$ summation formula is one of the most important results
in $q$-series, as presented in \cite[Eq. (2.7.1)]{Gasper-Rahman-04}: for $|aq/{bcd}|<1$,
\begin{align*}
{}_{6}W_{5}(a;b,c,d;q,aq/{bcd})
=\frac{(aq,aq/{bc},aq/{cd},aq/{bd};q)_\infty}{(aq/b,aq/c,aq/d,aq/{bcd};q)_\infty}.
\end{align*}
In this section, we will prove Theorem \ref{thm-cq-mn} by using the orthogonality relation
 for the continuous $q$-ultraspherical polynomials and the Rogers ${}_{6}\phi_{5}(\cdot)$ summation formula \cite[Eq. (2.7.1)]{Gasper-Rahman-04}.

\begin{proof}[The proof of Theorem \ref{thm-cq-mn}]
Using the generating function for the continuous $q$-ultraspherical polynomials in \eqref{defn-c-beta}
 and through  direct computation, we find that
\begin{align}\label{eq-cq-b-2}
\sum_{n=0}^{\infty}(1-\beta q^n)C_n(\cos\theta;\beta|q)t^n
=\frac{(\beta tqe^{i\theta},\beta tqe^{-i\theta};q)_\infty}
{(te^{i\theta},te^{-i\theta};q)_\infty}(1-\beta)(1-\beta t^2).
\end{align}
Multiplying both sides of \eqref{eq-cq-orth-b} by $(1-\beta q^n)t^n$
and then summing the resulting equation over $n$ from $0$ to $\infty$,
we can then apply the above equation to deduce that
\begin{align*}
&\int_{0}^{\pi}C_m(\cos\theta;\beta|q)\omega_\beta(\cos\theta|q)
\frac{(\beta tqe^{i\theta},\beta tqe^{-i\theta};q)_\infty}
{(te^{i\theta},te^{-i\theta};q)_\infty}(1-\beta t^2)d\theta\\
=&2\pi\frac{(\beta,q\beta;q)_\infty}{(q,\beta^2;q)_\infty}\frac{(\beta^2;q)_m}{(q;q)_m}t^m.
\end{align*}
Replacing $t$ with $tq^n$ in the above equation yields
\begin{align*}
&\int_{0}^{\pi}C_m(\cos\theta;\beta|q)\omega_\beta(\cos\theta|q)
\frac{(\beta tqe^{i\theta},\beta tqe^{-i\theta};q)_\infty}
{(te^{i\theta},te^{-i\theta};q)_\infty}
\frac{(te^{i\theta}, te^{-i\theta};q)_n}
{(\beta tqe^{i\theta},\beta tqe^{-i\theta};q)_n}(1-\beta t^2q^{2n})d\theta\\
=&2\pi\frac{(\beta,q\beta;q)_\infty}{(q,\beta^2;q)_\infty}\frac{(\beta^2;q)_m}{(q;q)_m}(tq^n)^m.
\end{align*}
Multiplying both sides of the above equation by
\begin{align*}
\frac{(q\beta/\gamma,\beta t^2;q)_n}{(1-\beta t^2)(q,\gamma t^2;q)_n}\gamma^n,
\end{align*}
and then summing the resulting equation from $n=0$ to $n=\infty$, we find that
\begin{align*}
&\int_{0}^{\pi}C_m(\cos\theta;\beta|q){}_{6}W_{5}(\beta t^2;q\beta/\gamma,te^{i\theta},te^{-i\theta};q,\gamma)
\frac{(\beta tqe^{i\theta},\beta tqe^{-i\theta};q)_\infty}
{(te^{i\theta},te^{-i\theta};q)_\infty}\omega_\beta(\cos\theta|q)d\theta\\
=&2\pi\frac{(\beta,q\beta;q)_\infty}{(q,\beta^2;q)_\infty}
\frac{(\beta^2;q)_m}{(1-\beta t^2)(q;q)_m}t^m
\sum_{n=0}^{\infty}\frac{(q\beta/\gamma,\beta t^2;q)_n}{(q,\gamma t^2;q)_n}(\gamma q^m)^n.
\end{align*}
Using the Rogers ${}_{6}\phi_{5}(\cdot)$ summation formula \cite[Eq. (2.7.1)]{Gasper-Rahman-04}, we deduce that
\begin{align*}
{}_{6}W_{5}(\beta t^2;q\beta/\gamma,te^{i\theta},te^{-i\theta};q,\gamma)
=\frac{(\beta q,q\beta t^2,\gamma te^{i\theta},\gamma te^{-i\theta};q)_\infty}
{(\gamma,\gamma t^2,\beta tqe^{i\theta},\beta tqe^{-i\theta};q)_\infty}.
\end{align*}
Combining the above two equations, we obtain that
\begin{align*}
&\int_{0}^{\pi}C_m(\cos\theta;\beta|q)
\frac{(\gamma te^{i\theta},\gamma te^{-i\theta};q)_\infty}
{(te^{i\theta}, te^{-i\theta};q)_\infty}\omega_\beta(\cos\theta|q)d\theta\\
=&2\pi\frac{(\beta,\gamma;q)_\infty}{(q,\beta^2;q)_\infty}
\frac{(\beta^2;q)_m}{(q;q)_m}t^m
\sum_{n=0}^{\infty}\frac{(q\beta/\gamma;q)_n(\gamma t^2q^n;q)_\infty}{(q;q)_n(\beta t^2q^n;q)_\infty}(\gamma q^m)^n.
\end{align*}
By the $q$-binomial theorem \eqref{defn-bi-thm}, we have
\begin{align*}
\frac{(\gamma t^2q^n;q)_\infty}{(\beta t^2q^n;q)_\infty}
=\sum_{m=0}^{\infty}\frac{(\gamma/\beta;q)_m}{(q;q)_m}(\beta t^2q^n)^m.
\end{align*}
Using the above equation and the $q$-binomial theorem \eqref{defn-bi-thm} again, we conclude that
\begin{align}\label{eq-beta-gamma}
\sum_{n=0}^{\infty}\frac{(q\beta/\gamma;q)_n(\gamma t^2q^n;q)_\infty}{(q;q)_n(\beta t^2q^n;q)_\infty}(\gamma q^m)^n
=&\sum_{j=0}^{\infty}\frac{(\gamma/\beta;q)_j}{(q;q)_j}(\beta t^2)^j\sum_{n=0}^{\infty}\frac{(q\beta/\gamma)_n}{(q;q)_n}(\gamma q^{m+j})^n\\
=&\sum_{j=0}^{\infty}\frac{(\gamma/\beta;q)_j(\beta q^{m+j+1};q)_\infty}{(q;q)_j(\gamma q^{m+j};q)_\infty}\beta^j t^{2j}\nonumber\\
=&\frac{(\beta;q)_\infty}{(\gamma;q)_\infty}\sum_{j=0}^{\infty}\frac{(\gamma/\beta;q)_j(\gamma;q)_{m+j}}{(q;q)_j(\beta;q)_{m+j+1}}\beta^j t^{2j}\nonumber.
\end{align}
It follows that
\begin{align*}
&\int_{0}^{\pi}C_m(\cos\theta;\beta|q)
\frac{(\gamma te^{i\theta},\gamma te^{-i\theta};q)_\infty}
{(te^{i\theta}, te^{-i\theta};q)_\infty}\omega_\beta(\cos\theta|q)d\theta\\
=&2\pi\frac{(\beta,\beta;q)_\infty}{(q,\beta^2;q)_\infty}
\frac{(\beta^2;q)_m}{(q;q)_m}t^m
\sum_{n=0}^{\infty}\frac{(\gamma/\beta;q)_j(\gamma;q)_{m+j}}{(q;q)_j(\beta;q)_{m+j+1}}\beta^j t^{2j+m}.
\end{align*}
Substituting
\begin{align*}
\frac{(\gamma te^{i\theta},\gamma te^{-i\theta};q)_\infty}
{(tqe^{i\theta}, tqe^{-i\theta};q)_\infty}
=\sum_{n=0}^{\infty}C_n(\cos\theta;\gamma|q)t^n
\end{align*}
into the above equation and interchanging $m$ and $n$, we conclude that
\begin{align*}
&\sum_{m=0}^{\infty}\left\{\int_{0}^{\pi}C_n(\cos\theta;\beta|q)C_m(\cos\theta;\gamma|q)\omega_\beta(\cos\theta|q)d\theta\right\}t^m\\
=&2\pi\frac{(\beta,\beta;q)_\infty(\beta^2;q)_n}{(q,\beta^2;q)_\infty(q;q)_n}
\sum_{j=0}^{\infty}\frac{(\gamma/\beta;q)_j(\gamma;q)_{n+j}}{(q;q)_j(\beta;q)_{n+j+1}}\beta^jt^{2j+n}.
\end{align*}
Equating the coefficients of power of $t$, we complete the proof of Theorem \ref{thm-cq-mn}.
\end{proof}

Using Theorem \ref{thm-cq-mn} and the orthogonality relation for the continuous $q$-ultraspherical polynomials,
 we can easily derive the following connection formula for the continuous $q$-ultraspherical polynomials
 due to Rogers \cite{Rogers-PLM-95}:
\begin{align*}
C_n(\cos\theta;\gamma|q)
=\sum_{k=0}^{\lfloor\frac{n}{2}\rfloor}\frac{\beta^k(\gamma/\beta;q)_k(\gamma;q)_{n-k}(1-\beta q^{n-2k})}{(q;q)_k(\beta q;q)_{n-k}(1-\beta)}C_{n-2k}(\cos\theta;\beta|q).
\end{align*}
Rogers \cite{Rogers-PLM-95} obtained this formula by determining the coefficients for small values of $n$,
and then proving it by induction. Askey and Ismail \cite{Askey-Ismail-83} proved this formula
by first computing an integral involving the product of the continuous $q$-ultraspherical polynomials
and the Chebyshev polynomials of the first kind. They subsequently applied the orthogonality relation
for the continuous $q$-ultraspherical polynomials and the Rogers ${}_{6}\phi_{5}(\cdot)$ summation formula \cite[Eq. (2.7.1)]{Gasper-Rahman-04}.

If we know the Rogers' connection formula in advance, we can easily prove Theorem \ref{thm-cq-mn}
by combining the Rogers' connection formula with the orthogonality relation
for the continuous $q$-ultraspherical polynomials.

For $k$ being a positive integer, by setting $m=2k+n$ in Theorem \ref{thm-cq-mn},
dividing both sides of the resulting equation by $2(1-\gamma)$,
and taking the limit as $\gamma\rightarrow1$ while using \cite[Eq. (3.14)]{Askey-Wilson-85},
we arrive at the following integral formula formula due to Askey and Ismail \cite[Eq. (4.12)]{Askey-Ismail-83}:
\begin{align*}
&\int_{0}^{\pi}C_n(\cos\theta;\beta|q)T_{n+2k}(\cos\theta)W_\beta(\cos\theta)d\theta\\
=&\pi\left[\begin{matrix}n+k\\k\end{matrix}\right]_q
\frac{(1-q^{n+2k})}{(1-q^{n+k})}\frac{(\beta,\beta q^{n+k+1};q)_\infty}{(q,\beta^2q^n;q)_\infty}\beta^k(\beta^{-1};q)_k,\quad k=1,2,\cdots,
\end{align*}
where $T_n(\cos\theta)=\cos {n\theta}$ are the Chebyshev polynomials of the first kind.

Using the Rogers ${}_{6}\phi_{5}(\cdot)$ summation formula \cite[Eq. (2.7.1)]{Gasper-Rahman-04},
we can prove the following proposition.
\begin{prop}\label{prop-cq-rela}
For $|t|<1$, we have
\begin{align*}
\sum_{n=0}^{\infty}C_n(\cos\theta;\gamma|q)t^n
=\sum_{j,m=0}^{\infty}\frac{(1-\beta q^m)(\gamma/\beta;q)_j(\gamma;q)_{j+m}\beta^j}{(q;q)_j(\beta;q)_{m+j+1}}C_m(\cos\theta;\beta|q)t^{m+2j}.
\end{align*}
\end{prop}
\begin{proof}
Replacing $t$ by $tq^m$ in the \eqref{eq-cq-b-2}, we have
\begin{align*}
&\sum_{n=0}^{\infty}(1-\beta q^n)C_n(\cos\theta;\beta|q)(tq^m)^n\\
=&\frac{(\beta tqe^{i\theta},\beta tqe^{-i\theta};q)_\infty}
{(te^{i\theta}, te^{-i\theta};q)_\infty}(1-\beta)(1-\beta t^2q^{2m})
\frac{(tqe^{i\theta},tqe^{-i\theta};q)_m}
{(\beta te^{i\theta}, \beta te^{-i\theta};q)_m}.
\end{align*}
Multiplying both sides of the above equation by
\begin{align*}
\frac{(q\beta/\gamma,\beta t^2;q)_m}{(q,\gamma t^2;q)_m}\gamma^m,
\end{align*}
and then summing the resulting equation from $m=0$ to $m=\infty$, we have
\begin{align*}
&\sum_{n=0}^{\infty}(1-\beta q^n)C_n(\cos\theta;\beta|q)t^n\sum_{m=0}^{\infty}\frac{(q\beta/\alpha,\beta t^2;q)_m}{(q,\gamma t^2;q)_m}(\gamma q^n)^m\\
=&(1-\beta)(1-\beta t^2)\frac{(\beta tqe^{i\theta},\beta tqe^{-i\theta};q)_\infty}
{(te^{i\theta}, te^{-i\theta};q)_\infty}
{}_6\phi_5\bigg(\genfrac{}{}{0pt}{}{\beta t^2,q\sqrt{\beta t^2},-q\sqrt{\beta t^2},q\beta/\gamma,te^{i\theta},te^{-i\theta}}{\sqrt{\beta t^2},-\sqrt{\beta t^2},\gamma t^2,\beta tqe^{i\theta},\beta tqe^{-i\theta}};q,\gamma\bigg).
\end{align*}
Using the Rogers ${}_{6}\phi_{5}(\cdot)$ summation formula \cite[Eq. (2.7.1)]{Gasper-Rahman-04}, we deduce that
\begin{align*}
{}_6\phi_5\bigg(\genfrac{}{}{0pt}{}{\beta t^2,q\sqrt{\beta t^2},-q\sqrt{\beta t^2},q\beta/\gamma,te^{i\theta},te^{-i\theta}}{\sqrt{\beta t^2},-\sqrt{\beta t^2},\gamma t^2,\beta tqe^{i\theta},\beta tqe^{-i\theta}};q,\gamma\bigg)
=\frac{(\beta q,q\beta t^2,\gamma te^{i\theta},\gamma te^{-i\theta};q)_\infty}{(\gamma,\gamma t^2,\beta tqe^{i\theta},\beta tqe^{-i\theta};q)_\infty}.
\end{align*}
Combining the above two equations, we arrive at
\begin{align*}
&\sum_{m=0}^{\infty}(1-\beta q^m)C_m(\cos\theta;\beta|q)t^m
\sum_{n=0}^{\infty}\frac{(q\beta/\gamma,\beta t^2;q)_n}{(q,\gamma t^2;q)_n}(\gamma q^m)^n\\
=&\frac{(\gamma tqe^{i\theta},\gamma tqe^{-i\theta};q)_\infty}{(te^{i\theta},te^{-i\theta};q)_\infty}
\frac{(\beta,\beta t^2;q)_\infty}{(\gamma,\gamma t^2;q)_\infty}.
\end{align*}
Applying the generating function for the continuous $q$-ultraspherical polynomials
from \eqref{defn-c-beta} to the right-hand side of the above equation, we deduce that
\begin{align*}
&\sum_{m=0}^{\infty}(1-\beta q^m)C_m(\cos\theta;\beta|q)t^m
\sum_{n=0}^{\infty}\frac{(q\beta/\gamma;q)_n(\gamma t^2q^n;q)_\infty}{(q;q)_n(\beta t^2q^n;q)_\infty}(\gamma q^m)^n\\
=&\frac{(\beta;q)_\infty}{(\gamma;q)_\infty}\sum_{n=0}^{\infty}C_n(\cos\theta;\gamma|q)t^n.
\end{align*}
Substituting \eqref{eq-beta-gamma} into the above equation, we complete the proof Proposition \ref{prop-cq-rela}.
\end{proof}

\subsection*{Acknowledgements}

The first author was partially supported  by  the National Natural Science Foundation of China (Grant No. 12201387). 
The second author was also partially supported  by National Natural Science Foundation of China (Grant No. 12371328).








\begin{thebibliography}{10}
\bibitem{AlSalam-Verma-PAMS-82}
W.A. Al-Salam and A. Verma, Some remarks on q-beta integral. Proc. Amer. Math. Soc. 85, 360–362 (1982)
\bibitem{Askey-Ismail-80}
R. Askey and M. Ismail, The Rogers q-ultraspherical polynomials, Approximation Theory. III
(E. W. Cheney, ed.), Academic Press, New York, 1980, pp. 175–182.
\bibitem{Askey-Ismail-83}
R. Askey and M. Ismail, A generalization of the ultraspherical polynomials, Studies in Pure
Mathematics (P. Erd\"{o}s, ed.), Birkh\"{a}user, Basel, 1983, pp. 55–678.
\bibitem{Askey-Wilson-85}
R. Askey and J. Wilson, Some basic hypergeometric orthogonal polynomials that generalize
Jacobi polynomials, Mem. Amer. Math. Soc. 54 (1985), no. 319, iv+55
\bibitem{Gasper-CM-81}
G. Gasper,
\emph{Orthogonality of certain functions with respect to complex valued weights.} Canadian J. Math. \textbf{33} (1981), no.5, 1261–1270. \MR{0638380}
\bibitem{Gasper-Rahman-04}
G. Gasper and M. Rahman, Basic Hypergeometric Series, Cambridge Univ. Press, Cambridge, MA, 2004.
\bibitem{Greiner-CMB-80}
P.C. Greiner,  Spherical harmonics on the Heisenberg group, Canadian Math. Bull. 23 (1980), 383-396.
\bibitem{Ismail-05}
Mourad E. H. Ismail, Classical and quantum orthogonal polynomials in one variable, with two
chapters by Walter Van Assche, with a foreword by Richard A. Askey, reprint of the 2005
original, Encyclopedia of Mathematics and its Applications, vol. 98, Cambridge University Press,
Cambridge, 2009.
\bibitem{KLS-10}
R. Koekoek, P.A. Lesky and R.F. Swarttouw, Hypergeometric orthogonal polynomials and their $q$-analogues, with a foreword by Tom H. Koornwinder, Springer Monographs in Mathematics, Springer-Verlag, Berlin, 2010.
\bibitem{Liu-17}
Z.-G. Liu, On a system of q-partial differential equations with applications to q-series. Analytic number theory, modular forms and q-hypergeometric series, 445–461, Springer Proc. Math. Stat., 221, Springer, Cham, 2017.
\bibitem{Rogers-PLM-95}
L.J. Rogers, Third memoir on the expansion of certain infinite products, Proc. London Math. Soc. 26 (1895) 15–32.
\end{thebibliography}
\end{document}